\documentclass{article}

\usepackage{mathptmx}       
\usepackage{helvet}         
\usepackage{courier}        
\usepackage{type1cm}        
\usepackage{makeidx}         
\usepackage{graphicx}        
\usepackage{multicol}        
\usepackage[bottom]{footmisc}

\usepackage{amsmath}
\usepackage{amssymb}
\usepackage{amsthm}

\newtheorem{theorem}{Theorem}
\newtheorem{lemma}{Lemma}

\numberwithin{theorem}{section}
\numberwithin{lemma}{section}
\numberwithin{figure}{section}

\newcommand{\rd}{\mathrm{d}}
\newcommand{\ri}{\mathrm{i}}
\newcommand{\RR}{\mathbb{R}}
\newcommand{\Span}{\text{Span}}
\newcommand{\Kn}{\mathsf{Kn}}

\graphicspath{{figures/}}

\makeindex             

\begin{document}

\title{On quantifying uncertainties for the linearized BGK
kinetic equation}
\author{C. Klingenberg, Q. Li, M. Pirner}
%
%
\date{}
\maketitle


\abstract{We consider the linearized BGK equation and want to quantify uncertainties in the case of modelling errors. More specifically, we want to quantify the error produced if the pre-determined equilibrium function is chosen inaccurately. In this paper we consider perturbations in the velocity and in the temperature of the equilibrium function and consider how much the error is amplified in the solution.}

\section{Introduction}
Kinetic equation is a set of integro-differential equations that describe the collective behavior of many-particle systems. The to-be-solved unknown function is a probability distribution of particles defined on the phase space, and kinetic equation characterizes its evolution in time and space. The equation typically has one transport term representing the movement of particles and one collision operator that describes the interactions between particles. The specific form of the transport and the collision operators depend on the system one is looking at. Typically people use radiative transfer equation for photon particles, the Boltzmann equation for rarified gas particles, the Fokker-Planck equation for plasma, and run-and-tumble models for bacteria. There are many more other examples.

Uncertainty is a nature of kinetic theory. It has various of origins. The forms of terms in the equation are usually unjustified due to the modeling error, the blurred measurements are typically not enough to sufficiently determine the coefficients, and the initial and boundary conditions are never provided as accurate as they are supposed to be. They all contribute the inaccuracy of the system description. It is not realistic to look for the most accurate description of systems, nor expect the exact true solution, and thus we instead look for possibilities of quantifying the uncertainties, and ask if the error is controllable even if the models and measurements are not accurate. As presented above there are many origins of error, and in this paper we focus on the modeling error. More specifically, a typical way of simplifying kinetic equations is to perform linearization around a pre-determined equilibrium function and compute the linearized kinetic equation, and we would like to understand the error produced if the pre-determined equilibrium function is chosen inaccurately. We plan to answer this question from both analytical point of view and numerical point of view. In particular we would like to understand that given certain perturbation on the pre-determined equilibrium where we perform the linearization, by how much the error is amplified in the solution, and how to characterize the perturbation numerically.

There have been many numerical techniques that were developed to address uncertainties. One very popular category of methods are termed generalized polynomial types. These include generalized polynomial chaos method (gPC)~\cite{GHANEM1996289,ghanem_construction_2006,Xiu_gpc, XK02}, and stochastic collocation method~\cite{babuska_stochastic_2007,XiuHesthaven_collocation}. These methods assume the uncertainties in the parameters of the equations are reflected as a polynomial type in the solution. And based on this assumption one applies either the spectral method, or the psudo-spectral method, and expand the solution in the random direction using polynomials. Another popular, or even classical method is the Monte Carlo type method, which also has many variations~\cite{fishman2013monte,Giles_MLMC,Barth2011,Charrier_MLMC}. With these methods one simply samples the random variable many times, and for each sample the parameters are fixed and the equation is considered deterministic, and one computes the equation. In the end one ensembles the solutions for the mean and the variance. Sometimes mathematicians categorize these methods based on if new implementations are needed. Since the Monte Carlo type method and stochastic collocation method simply call the deterministic solver many times, the old algorithms are therefore recycled and they are categorized as non-intrusive methods, while on the other hand, the traditional generalized polynomial chaos method is intrusive, wherein a completely new implementation is needed. In terms of the convergence rate, it is well-known that the Monte Carlo method converges slowly, while the gPC type methods are spectral types along the random directions, and automatically inherit the so-called spectral convergence: depending on the regularity of the solution in the random space, the method could be either algebrically fast or exponentially fast.

We would like to adopt the gPC framework for its possible fast convergence. To do that, in our setting, we mainly need to prove that the perturbation in the solution continuously depends on the perturbation in the equilibrium function where we choose to perform linearization. According to the standard spectral method theory, the higher degree of continuity means the faster convergence. Traditionally, this framework has been successfully applied in treating elliptic type equation~\cite{BNT07,BNT04, ZG12, cohen_analytic_2011, cohen_convergence_2010}, and the analysis sometimes even suggests new algorithms that better explore the solution structure~\cite{hou_li_zhang16_1,hou_li_zhang16_2,chkifa_breaking_2014,NTW08,nobile_sparse_2008,nobile_anisotropic_2008,SchwabTodor_sparse, GPZ15,DP_uncertaity}, but when applied onto hyperbolic type equations, this framework sees limited success due to the intrinsic difficulties~\cite{majda_gpc,DP_uncertaity}: the solution develops non-smooth structure, breaking the assumptions the spectral methods rely on.

The standard kinetic equation does not belong to either of the category mentioned above but could produce both. Depending on the regime one is interested in, kinetic equation would either converge to a hyperbolic type (such as BGK equation converging to the Euler equation) or a parabolic type (such as radiative transfer equation converging to the heat equation). On one hand, its transport term represents hyperbolic type and shows a traveling wave behavior, in the meantime, the collision term in kinetic equations are all coercive terms and thus provide some dissipative behavior and represents the parabolic type. This unique feature presents mathematicians a new world to explore and it indeed triggers many studies recently. Some recent results on the topic can be found in~\cite{JXZ15, HJ16,JLu16,JZ16, JLM16,JL16,LW16}. We have to mention, however, most of the proofs are accomplished on a case-by-case basis, and not necessarily in their sharpest estimates, especially in the big space long time regime, except in~\cite{LW16} where the authors started with an abstrat form and were able to employ the hypocoercivity for a uniform bound across regimes.

Follow the previous work, in this paper we explore the perturbation on the linearization point. We take the BGK equation as a starting point and perturb $u$, the bulk velocity, and $T$, the temperature in the equilibrium function, by $z$, a random variable. The domain of $z$ indicates the strength of the perturbation. And we would like to study how $f$, the solution to the linearized equation, respond to the variations in $z$.

We lay out the equation and its basic assumptions in Section 2, together with detailed studies of the convergence rate in time in the deterministic setting. Section 3, 4 and 5 are respectively devoted to the study extended to equations in various of regimes,  to equations involving randomness, and to scenarios when both present. We conclude in Section 6.

\section{Set-up}
The BGK equation, known as a simplified model of the Boltzmann equation, writes as:
\begin{equation}
\partial_t F + v\cdot\nabla_xF = \frac{1}{\Kn}\left( M[F] - F \right)
\end{equation}
where $F(t,x,v)$ is the distribution function living on phase space indicating the distribution of rarified gas. $M[F]$, the so-called Maxwellian function, is a Gaussian distribution function:
\begin{equation}
M[F] = \frac{\rho}{(2\pi T)^{d/2}} \exp^{-\frac{|v-u|^2}{2T}}\,,
\end{equation}
with its macroscopic quantities defined implicitly by $F$ such that the first $d+2$ moments are the same:
\begin{equation}
\int \phi(M[F] - F) \rd{v} = 0\,,
\end{equation}
with $\phi = [1,v,v^2]^T$. This property is typically called conservation property, since it immediately leads to density, momentum and energy conservation:
\begin{equation}
\partial_t\int\phi F\rd{v} + \nabla_x\int v\otimes vF\rd{v} = 0\,.
\end{equation}
If we use the definition:
\begin{equation}
\int F\rd{v} = \rho(x)\,,\quad\int vF\rd{v} = \rho(x)u(x)\,,\quad\text{and}\quad\int \frac{|v|^2}{2}F\rd{v} = E = \frac{1}{2}\rho u^2+\rho T\,.
\end{equation}
then the first two equations express the conservation law of the density and momentum. Note that second term in the last equation cannot be presented using any macroscopic quantities and thus the system is not closed.

$\Kn$ is termed the Knudsen number. It comes from rescaling the system by setting $t\to\frac{t}{\Kn}$ and $x\to\frac{x}{\Kn}$. When $\Kn$ is small, the system is seen in large domain and long time scale and falls in the hyperbolic regime. More specifically, as $\Kn\to 0$, the leading term in the equation reads:
\begin{equation}
\frac{1}{\Kn}\left(M[F]-F\right) = 0\quad\Rightarrow\quad F = M[F]\,,
\end{equation}
and thus $\int v|v|^2F\rd{v}$ could be explicitly expressed and we rewrite equation as:
\begin{equation}
\left\{\begin{array}{l}\partial_t\rho +\nabla_x\cdot(\rho u) = 0\\ \partial_t\rho u +\nabla_x(\rho u\otimes u+\rho T) = 0\\ \partial_t E + \nabla_x\left((E+\rho T)u\right)= 0\end{array}\right.
\end{equation}

For linearization we typically assume the solution is close enough to a particular Maxwellian, meaning there exists $f$ and $M_\ast$ such that:
\begin{equation}
F = (1+f)M_\ast\,,\quad\text{with}\quad |f|\ll 1\,.
\end{equation}

Plug this ansatz back into the full BGK equation and ignore the higher order expansion terms, we have:
\begin{equation*}
\partial_tf + v\cdot\nabla_xf = \frac{1}{\Kn}\mathcal{L}_\ast f = \frac{1}{\Kn}\left(m[f] - f\right)\,,
\end{equation*}
where $m$ is a quadratic function that shares the same moments with $f$, meaning:
\begin{equation}
\langle \phi\,,m-f\rangle_\ast = \int\left(\begin{array}{c}1\\v\\v^2\end{array}\right)(m[f]-f)M_\ast\rd{v} = 0\,.
\end{equation}
Here we used the definition of the inner product:
\begin{equation}
\langle f,g\rangle_\ast = \int fgM_\ast\rd{v}\,.
\end{equation}
This is the counterpart of the conservation law in linearized system since:
\begin{equation}
\partial_t\int\phi fM_\ast\rd{v} + \nabla_x\int v\otimes vfM_\ast\rd{v} = 0\,.
\end{equation}
Once again if $\Kn$ is small then in the leading order $f = m$ which leads to a closed Euler system, termed acoustic limit:
\begin{equation}
	\partial_t U+A\cdot\partial_x U=0\,.
\end{equation}
Here
\begin{equation}
A=\left(\begin{array}{ccc} u_* & \rho_* & 0 \\
        \frac{T_*}{\rho_*} & u_* & 1 \\
        0 & 2T_* & u_* \\
      \end{array}\right)\,,\quad\text{and}\quad U=[\tilde{\rho},\tilde{u},\tilde{T}]^T\,,
\end{equation}
and the macroscopic quantities are defined by:
\begin{equation}
  \int f\left(\begin{array}{c}1 \\v \\v^2\end{array}\right)\rd{v}=\left(\begin{array}{c}
\tilde{\rho} \\\tilde{\rho}u_*+\rho_*\tilde{u} \\\tilde{\rho}(u_*^2+T_*)+2\rho_* u_*\tilde{u}+\rho_*\tilde{T}\end{array}\right)\,.
\end{equation}

There are several very well-known properties of the linear operator:
\begin{itemize}
\item[1]{Coercive:} $\langle\mathcal{L}_\ast f\,,f\rangle_\ast\leq 0$\,,
\item[2]{Explicit null space:} $\mathcal{L}_\ast f = 0\quad f \in\Span\{1,v,v^2\}$,
\item[3]{Self-adjoint:} $\langle \mathcal{L}_\ast f\,,g\rangle_\ast = \langle f\,,\mathcal{L}_\ast g\rangle_\ast$.
\end{itemize}
Combining item $2$ and $3$ it is easy to see $\langle \mathcal{L}_\ast f\,,\phi\rangle_\ast = 0$.
If we consider $f\in L_2(M_\ast\rd{v})$, one could express $\mathcal{L}_\ast$ more explicitly. By the definition of $m[f]$ it is easy to see it is in fact a projection of $f$ weighted by $M_\ast$ on the quadratic function space:
\begin{equation}
\mathcal{L}_\ast f = m - f = \Pi_\ast{f} - f\,,\quad\text{with}\quad \Pi_\ast{f} = \sum_{i=0}^{d+1}\langle\chi_i\,,f\rangle_\ast\chi_i\,,
\label{eqn:L}
\end{equation}
where $\chi_i$ are basis functions satisfying:
\begin{itemize}
\item[1]{Expand the space} $\Span\{\chi_m,m=0,\cdots d+1\} = \Span\{1,v,v^2\}$,
\item[2]{Orthogonality} $\langle \chi_m\,,\chi_n\rangle_\ast = \delta_{mn}$.
\end{itemize}

With the Maxwellian function $M_\ast$ predetermined, they are simply the first $d+2$ Hermite polynomials associated with the Maxwellian. Even more if we set $\chi_m$ the $m$-th Hermite polynomial for all $m$, then
\begin{equation}
\mathcal{L}_\ast{f} = - \sum_{m=d+2}^\infty\langle\chi_m\,,f\rangle_\ast\chi_m\,.
\end{equation}
This expression also explicitly suggests the coercivity of the operator.

The linearized BGK operator has been studied by many researchers. Serving as the simplied version of the linearized Boltzmann equation. Its negative spectrum provides dissipative behavior, which helps us in getting existence and uniqueness of the solution at ease. In the boundary layer analysis, the nonlinear collision operator is far from being understood, the linearized equation is the stepping stone for connecting the Dirichet data for the kinetic and the Dirichlet data for the interior Euler equation. We mention several recent work on boundary layer analysis for the linearized BGK equation here~\cite{LiLuSun2016_JSP,LiLuSun2016_half,LiLuSun2015_general,LiLuSun2015JCP}.

However, all these studies are based on the assumption that the Maxwellian $M_\ast$, the function we linearize upon, is given a priori, which is typically not the case. Taking numerical algorithm provided in~\cite{LiLuSun2015JCP} for example, we choose to perform linearization upon the Maxwellian function provided from the previous time step as an approximation to the true Maxwellian, which is in fact at least $\mathcal{O}(\Delta t)$ away from the real Maxwellian. A natural question one needs to address there is: is such approximation a good approximation, or rather, if the Maxwellian chosen is off from the accurate one by $\mathcal{O}(\Delta t)$, how much error does $f$ contain.

Since $M_\ast$'s dependence on $\rho_\ast$ is linear, and thus its reflection in $f$ is of less interest. We in this paper only study the possible deviation of the solution $f$ when $M_\ast$ has a uncertain $u_\ast$ and a uncertain $T_\ast$.

\section{Variation in $u$}
In this section we study the solution's response to deviations in $u_\ast$. We firstly repeat the equation in 1D:
\begin{equation*}
\begin{cases}
\partial_tf + v\partial_xf = \mathcal{L}_\ast f\,,\quad (t,x,v)\in[0,\infty)\times\RR\times\RR\\
f(t=0,x,v) = f_\ri(x,v)
\end{cases}\,,
\end{equation*}
with $\mathcal{L}_\ast f = m-f$ such that $\langle\phi\,,m-f\rangle_\ast = 0$, and $f_\ri$ is the initial data. Assume the Maxwellian:
\begin{equation}
M_\ast = \frac{\rho_\ast}{\sqrt{2\pi T_\ast}}\exp{\left(-\frac{|v-u_\ast|^2}{2T_\ast}\right)}
\label{eqn:M}
\end{equation}
 and assume that $f$ decays fast enough to zero as $x \rightarrow \infty$ such that $\langle \partial_x f, f \rangle_x=0$. \\

with $u_\ast(z)$ depending on a random parameter $z$\footnote{for practical purpose the range of $z$ is controlled by $\Delta t$ but we study the general case here}. We would like to understand the regularity of the solution $f$ on $z$ direction, namely we need to find a good bound for $\partial_zf$ in certain norm.

The standard way of pursuing such analysis is simply to take the derivative of $z$ on the entire equation for a equation for $\partial_zf$, and then study the bound of $\partial_zf$. The bound could serve as a Lipschitz constant, and if small, numerical solvers that require certain regularities could be applied. Sometimes people go beyond the first derivative and seek for high differentiation, and they are all bounded in a reasonable way, spectral method could be proved to be a effective method.

If we follow that procedure, however, the difficulty would be immediate: the random variable's dependence is hidden in the operator through $\mathcal{L}_\ast$ in a very subtle way. That means taking $z$ derivative of the whole equation will produce very complicated formulation on the right hand side. We thus choose a easy way that overcomes it by shifting the coordinates. Define
\begin{equation}
g(t,x,v) = f(t,x,v-u_\ast)\,,
\end{equation}
then the equation for $g$ will have a trivial collision but a shifted transport term:
\begin{equation}\label{eqn:g}
\begin{cases}
\partial_tg + (v+u_\ast)\partial_xg = \mathcal{L}_0 g\\
g(t=0,x,v) = g_\ri(x,v) = f_\ri(x,v-u_\ast)
\end{cases}\,,
\end{equation}
with $\mathcal{L}_0$ being associated with the Maxwellian with zero velocity. The $z$ dependence of the two functions could be easily written down:
\begin{equation}
\partial_zg = \partial_zf - \partial_vf\partial_zu_\ast\,, \quad\text{or}\quad \partial_zg + \partial_vg\partial_zu_\ast = \partial_zf\,.
\end{equation}
Since $\partial_vf$ is more understood, for now we focus on studying $\partial_zg$. We take the derivative of the entire equation to get:
\begin{equation*}
\partial_t\partial_zg + (v+u_\ast)\partial_x \partial_z g +\partial_zu_\ast\partial_xg = \mathcal{L}_0 \partial_zg\,,
\end{equation*}
or by defining $h = \partial_zg$ and reorganize the equation:
\begin{equation}\label{eqn:h}
\partial_th + (v+u_\ast)\partial_xh = \mathcal{L}_0 h-\partial_zu_\ast\partial_xg \,.
\end{equation}
Immediately we see that $h$ satisfies also the linearized BGK equation but has one more negative source term $-\partial_zu_\ast\partial_xg$ compared with~\eqref{eqn:g}. To have a certain bound of $h$, we mainly need to go through two steps:
\begin{itemize}
\item[1]{bound the source term:} one needs to prove that the source term $\partial_zu_\ast\partial_xg$ is bounded;
\item[2]{bound $h$ itself:} here we need to show that a bounded $\partial_xg$ will produce a bounded $h$.
\end{itemize}
These two statements are summarized in the following two theorems.

\begin{theorem}\label{thm:bound_partial_x_g}
$\|\partial_xg\|_2$ is bounded. More specifically: $$\|\partial_xg\|_{L^2(\rd{x}\rd{v})}(t) \leq\|\partial_xg_\ri\|_{L^2(\rd{x}\rd{v})}$$.
\end{theorem}
\begin{proof}
To show this we first write down the equation for $\partial_xg$. Take the derivative of Equation~\eqref{eqn:g} with respect to $x$ one gets:
\begin{equation}\label{eqn:partial_x_g}
\begin{cases}
\partial_t\partial_xg + (v+u_\ast)\partial^2_xg = \mathcal{L}_0 \partial_xg\\
\partial_xg(t=0,x,v) = \partial_xg_\ri(x,v)
\end{cases}\,.
\end{equation}
Here we note that $\mathcal{L}_0$ is an operator on $\rd{v}$ and commute with $\partial_x$. It immediately suggests that $\partial_xg$ satisfies the same equation as $g$ in~\eqref{eqn:g}. Considering that the linearized BGK equation is a dissipative system and the $L_2$ norm decays in time, we cite the following lemma:
\begin{lemma}\label{lemma:g}
Suppose $g$ satisfies equation~\eqref{eqn:g}, then
\begin{equation}
\|g\|_{L^2(\rd{x}\rd{v})}(t) \leq \|g_\ri\|_{L^2(\rd{x}\rd{v})}
\end{equation}
where $g_\ri$ is the initial condition.
\end{lemma}
\begin{proof}
The proof is based on energy estimate. We multiply the equation by $g$ and integrate with respect to $x$ and $v$, then:
\begin{equation}
\langle \partial_t g\,,g\rangle_{x,v} + \langle v\partial_xg\,,g\rangle_{x,v} = \langle\mathcal{L}_0g\,,g\rangle_{x,v}\,.
\end{equation}
Since we are considering the Cauchy problem we throw the second term away. The term on the right hand side is negative considering the coercivity of the collision operator. We then immediately get $\partial_t\langle g\,,g\rangle_{x,v}\leq 0$, meaning the $L_2$ norm of $g$ decays in time and thus:
\begin{equation}
\|g\|_{L^2(\rd{x}\rd{v})}(t) \leq \|g_\ri\|_{L^2(\rd{x}\rd{v})}\,.
\end{equation}
$\hfill\Box$
\end{proof}

Apply this lemma on~\eqref{eqn:partial_x_g} we conclude with Theorem~\ref{thm:bound_partial_x_g}.
$\hfill\Box$
\end{proof}

With the boundedness of the source term $\partial_zu_\ast\partial_xg$, we could start analyzing the bound for $h$. 
\begin{theorem}\label{thm:h}
Suppose $h = \partial_zg$ satisfies~\eqref{eqn:h}, then $\|h\|_{L^2(\rd{x}\rd{v})}$ grows at most linearly:
\begin{equation}
\|h\|_{L^2(\rd x\rd v)}\lesssim C\|\partial_xg_\ri\|_{L_2(\rd{x}\rd{v})}t\,.
\end{equation}
\end{theorem}
Here $f\lesssim g$ means $\frac{f}{g}$ is bounded by a constant in large time. We care only about the long time behavior of the solution. The reason is that after order one time, the highest order polynomial in time dominates the lower orders, and thus one only needs to specify the highest order coefficient. 

\begin{proof}
It is once again energy method. We multiply~\eqref{eqn:h} on both sides with $h$ and take the inner product in $(x,v)$:
\begin{equation}
\langle\partial_t h\,,h\rangle_{x,v} = \langle\mathcal{L}_0h\,,h\rangle_{x,v} -\partial_zu_\ast \langle\partial_xg\,,h\rangle_{x,v}\,.
\end{equation}
Considering the coercivity of $\mathcal{L}_0$ the first term on the right disappear. And we use Cauchy-Schwartz inequality to control the second term to get:
\begin{equation}
\frac{1}{2}\frac{d}{dt}\|h\|^2_{L^2(\rd{x}\rd{v})} \leq \|\partial_zu_\ast\partial_xg\|_{L^2(\rd{x}\rd{v})}\|h\|_{L^2(\rd{x}\rd{v})}\,.
\end{equation}
Assume $|\partial_zu_\ast|<C$, and it is known from Theorem~\ref{thm:bound_partial_x_g} that  $$\|\partial_xg\|_{L^2(\rd{x}\rd{v})}\leq\|\partial_xg_\ri\|_{L^2(\rd{x}\rd{v})},$$ then
\begin{equation}
\frac{d}{dt}\|h\|_{L^2(\rd{x}\rd{v})} \leq C\|\partial_xg_\ri\|_{L^2(\rd{x}\rd{v})}
\end{equation}
which leads to a linear growth of $h$: $\|h\|_{L^2(\rd{x}\rd{v})} \lesssim C\|\partial_xg_\ri\|_{L^2(\rd{x}\rd{v})}t$.
$\hfill\Box$
\end{proof}

The theorem above states the bounded of the first derivative of $g$ in $z$. One could extend it to treat higher order derivatives.
\begin{theorem}
Denote $h^{(n)} = \partial^n_zg$, then $\|h^{(n)}\|_{L^2(\rd{x}\rd{v})}$ is bounded by $t^n$:
\begin{equation}
\|h^{(n)}\|_{L^2(\rd v\rd x)}\lesssim C_nt^n\,.
\end{equation}
\end{theorem}
Again we are mainly interested in the long time behaviour of the solution so it suffices to consider only the highest order in time.
\begin{proof}
The proof is based on induction. According to the definition, $h^{(0)} = g$ and Lemma~\ref{lemma:g} guarantees that $h^{(0)}$ is bounded by a constant, and $h^{(1)}$ is the $h$ in Theorem~\ref{thm:h} and we have seen it is bounded by a linear growth. We thus perform math induction, assuming $h^{(k-1)}$ is bounded by $t^{k-1}$ we show that $h^{(k)}$ is bounded by $t^{k}$.

We first take the $k$-th order derivative of the equation~\eqref{eqn:g}:
\begin{equation*}
\partial_t\partial_z^kg + \sum_{n=0}^k{k\choose n}\partial_z^n(v+u_\ast)\partial_x\partial_z^{k-n}g = \mathcal{L}_0 \partial_z^k g\,,
\end{equation*}
or moving the source term to the right:
\begin{equation*}
\partial_th^{(k)} + (v+u_\ast)\partial_xh^{(k)} = \mathcal{L}_0 h^{(k)} - \sum_{n=1}^k{k\choose n}\partial_z^nu_\ast\partial_xh^{(k-n)}\,.
\end{equation*}
According to our assumption, $\partial_z^nu_\ast$ is bounded by a constant, one has:
\begin{equation*}
\begin{split}
\langle \partial_t h^{(k)}\,,h^{(k)}\rangle_{x,v} + \langle (v+u_\ast)\partial_xh^{(k)}\,,h^{(k)}\rangle_{x,v} &= \langle \mathcal{L}_0 h^{(k)}\,,h^{(k)}\rangle_{x,v} \\&- \sum_{n=1}^k{k\choose n}\langle \partial_z^nu_\ast\partial_xh^{(k-n)}\,,h^{(k)}\rangle_{x,v}\,.
\end{split}
\end{equation*}
which means:
\begin{equation}\label{eqn:bound_hk}
\frac{1}{2} \frac{d}{dt}\|h^{(k)}\|^2_{L_2(\rd{x}\rd{v})} \leq C_k\|\partial_xh^{(k-n)}\|_{L_2(\rd{x}\rd{v})}\|h^{(k)}\|_{L_2(\rd{x}\rd{v})}\,,
\end{equation}
where we used the Cauchy boundary condition, the coercivity of $\mathcal{L}_0$, and Cauchy-Schwartz inequality. By our assumption $h^{(k-1)}$ is bounded by $t^{k-1}$, since $\partial_xh$ and $h$ satisfies the same equation, it can be extrapolated as $\partial_xh$ being bounded by the same order, and then putting it back into~\eqref{eqn:bound_hk}, we have:
\begin{equation}
\|h^{(k)}\|_{L_2(\rd{x}\rd{v})}\lesssim t^{k}\,,
\end{equation}
which finishes the math induction loop, and complete the proof.
$\hfill\Box$
\end{proof}

\section{Variation in $T$}
In this section we want to study the solution's response to the deviations in $T_\ast$. Namely, we assume the Maxwellian defined in~\eqref{eqn:M} has its $T_\ast(z)$ depending on a random parameter $z$. Once again, in order to get rid of the complicated dependence of $\mathcal{L}_\ast$ on $z$, we perform change of variable and define 
\begin{equation}
p(t,x,v) = f\left(t,x, \frac{v}{\sqrt{T_\ast}}\right)\,.
\end{equation}
Then $p$ satisfies the equation
\begin{equation}
\begin{cases} \partial_t p + \sqrt{T_\ast} v \partial_xp = \mathcal{L}_1 p \\ p(t=0, x,v) = p_\ri(x,v) = f_\ri\left(x, \frac{v}{\sqrt{T_\ast}} \right) \end{cases}\,,
\label{eqn:p}
\end{equation}
where $\mathcal{L}_1$ is the collision operator associated with the Maxwellian with temperature one, and $p_\ri$ is the initial data. Again we focus on studying $\partial_z p$ instead of $\partial_z f$. Denote $q=\partial_z p$, we obtain its governing equation by taking the derivative in $z$ of equation~\eqref{eqn:p}. Rearranging the terms we have:
\begin{equation}
\partial_t q + \sqrt{T_\ast} v \partial_x q = \mathcal{L}_1 q - \partial_z(\sqrt{T_\ast}) v \partial_x p \,.
\label{eqn:q}
\end{equation}
This equation has the same structure as equation \eqref{eqn:h}: it is a linearized kinetic equation with a source term, and for the boundedness of $q$, we simply need to show the boundedness of $v \partial_x p$. In the previous section we showed that the source term $\partial_xg$ satisfies the same equation as $g$ does and thereby was able to give the bound. This is no longer the case here. Instead of writing the equation we write:
\begin{equation}
v \partial_x p = \frac{\mathcal{L}_1 p - \partial_t p}{\sqrt{T_\ast}}\,,
\label{eqn:v_partial_x_p}
\end{equation}
and are able to prove the following:
\begin{theorem}
Suppose $q= \partial_z p$ satisfies \eqref{eqn:q}, then $|| q||_{L^2(\rd{x}\rd{v})}$ grows at most linearly:
\begin{equation*}
||q||_{L^2(\rd{x}\rd{v})} \lesssim C \left( || p_i||_{L^2(\rd{x}\rd{v})} + || \partial_t p_i||_{L^2(\rd{x}\rd{v})} \right) t
\end{equation*}
\end{theorem}
\begin{proof}
We once again use the energy method. We insert \eqref{eqn:v_partial_x_p} into \eqref{eqn:q} and multiply the obtained equation with $q$ and take the inner product in $(x,v)$:
\begin{equation}
 \langle \partial_t q\,,q\rangle_{x,v} = \langle \mathcal{L}_1 q\,,q\rangle_{x,v} - \frac{\partial_z \sqrt{T_\ast}}{\sqrt{T_\ast}} \langle (\mathcal{L}_1p - \partial_t p)\,,q\rangle_{x,v}.
\end{equation}
Due to the coercivity of $\mathcal{L}_1$ the first term on the right disappears. For the second term on the right we use Cauchy-Schwarz and the triangle inequality
\begin{equation}
\frac{1}{2} \frac{d}{dt} ||q||^2_{L^2(\rd{x}\rd{v})} \leq \bigg \vert\frac{\partial_z \sqrt{T_\ast}}{\sqrt{T_\ast}} \bigg \vert \left( || \mathcal{L}_1 p ||_{L^2(\rd{x}\rd{v})} + || \partial_t p ||_{L^2(\rd{x}\rd{v})} \right) ||q||_{L^2(\rd{x}\rd{v})}
\end{equation}
We assume $\bigg \vert\frac{\partial_z \sqrt{T_\ast}}{\sqrt{T_\ast}} \bigg \vert < C$. 
 Similar to $\Pi_x f$ in \eqref{eqn:L}, $\Pi_1 g$ can be represented as $$\Pi_1 f = \sum_{i=0}^{d+1} \langle \chi_i^0, g \rangle_1 \chi_i^0$$ with orthonormal basis fuctions $\chi_i^0$, where $\langle \cdot \rangle_1$ denotes integration with respect to $v$ with the weight $M_1$. Then  $||\mathcal{L}_1 p||_{L^2(M_1 \rd{x}\rd{v})}$ can be estimated  by above using the explicit expression of $\mathcal{L}_1$  and Cauchy-Schwartz inequality by
\begin{equation}
\begin{split}
\langle \mathcal{L}_1 p , \mathcal{L}_1 p \rangle_{x,v,1} &= \langle \sum_{i=1}^{d+1} \langle \chi_i, p \rangle_{x,v,1}~ \chi_i - p , \sum_{j=1}^{d+1} \langle \chi_j, p \rangle_{x,v,1}~ \chi_j - p \rangle_{x,v,1} \\ &= \sum_{i=1}^{d+1} (\langle \chi_i , p \rangle_{x,v,1})^2 - \langle p,p\rangle_{x,v,1} \leq (d+1) \langle p,p \rangle_{x,v,1} - \langle p, p\rangle_{x,v,1}\\ & = d \langle p,p\rangle_{x,v,1}
\end{split}
\end{equation}
Since the norm $|| \cdot ||_{L^2(M_1\rd{x}\rd{v})}$ is equivalent to $|| \cdot ||_{L^2(\rd{x}\rd{v})}$, the term $|| \mathcal{L}_1 p||_{L^2(\rd{x}\rd{v})}$ is also bounded by $C ||p||_{L^2(\rd{x}\rd{v})}$.


 Realizing that  $\partial_t p$ satisfies the same equation as $p$ does, according to Lemma~\eqref{lemma:p}, their $L_2$ norm decrease in time, meaning:
\begin{align}
\frac{d}{dt}|| q||_{L^2(\rd{x}\rd{v})} &\leq C \left( ||  p||_{L^2(\rd{x}\rd{v})}(t) + || \partial_t p||_{L^2(\rd{x}\rd{v})}(t) \right) 
\\ &\leq C \left( ||  p_i||_{L^2(\rd{x}\rd{v})} + || \partial_t p_i||_{L^2(\rd{x}\rd{v})} \right)\,,
\end{align}
which leads to a linear growth of $q$:
 \begin{equation}
 || q||_{L^2(\rd{x}\rd{v})} \lesssim C \left( || p_i||_{L^2(\rd{x}\rd{v})} + || \partial_t p_i||_{L^2(\rd{x}\rd{v})} \right) t\,.
\end{equation}
which concludes the proof.
$\hfill\Box$
\end{proof}

The lemma used in the theorem is stated in the following:
\begin{lemma}\label{lemma:p}
Suppose $p$ satisfies equation~\eqref{eqn:p}, then
\begin{equation}
\|p\|_{L^2(\rd{x}\rd{v})}(t) \leq \|p_\ri\|_{L^2(\rd{x}\rd{v})}
\end{equation}
where $p_\ri$ is the initial condition.
\end{lemma}
\begin{proof}
The proof is analogous to the proof of Lemma 3.1.
$\hfill\Box$
\end{proof}

We can also extend the result of Theorem 4.1 to derivatives of higher orders. This is done in the following theorem
\begin{theorem}
Suppose $q^{(n)}:= \partial_z^n p$ satisfies
\begin{equation}
 \partial_t q^{(n)} + \sum_{k=0}^n \begin{pmatrix}
 ~n~ \\ ~k~
\end{pmatrix} \partial_z^{(n-k)} \left(\sqrt{T_\ast}\right) v \partial_x q^{(k)} = \mathcal{L}_1 q^{(n)} 
\label{eqn:q^n}
\end{equation}
for all $n \in \mathbb{N}_0$. Then 
\begin{equation*}
||q^{(N)}||_{L^2(\rd{x}\rd{v})} \lesssim C_N ~ t^N \quad \text{for all} ~ N\leq n
\end{equation*}
 where $C_N$ depends on $|| \mathcal{L}_1^k \partial_t^l q^{(0)}||_{L^2(\rd{x}\rd{v})}, ~ k,l \leq N$.
\end{theorem}
\begin{proof}
We proof the statement via induction. For $n=0$ and $n=1$ we proved it in Theorem 4.1 and Lemma 4.1. We have shown in Lemma 4.1 that if $q^{(0)}$ satisfies \eqref{eqn:p}, then $||q^{(0)}||_{L^2(\rd{x}\rd{v})}$ is bounded by $||q^0_i||_{L^2(\rd{x}\rd{v})}$, and in Theorem 4.1 if $q^{(1)}$ satisfies \eqref{eqn:q}, we can replace $v \partial_x q^{(0)}$ by \eqref{eqn:v_partial_x_p}. We can show that $\mathcal{L}_1 q^{(0)}$ is bounded in ${L^2(\rd{x}\rd{v})}$ by $||p||_{L^2(\rd{x}\rd{v})}$ and $\partial_t q^{(0)}$ also satisfy \eqref{eqn:p} and can deduce that $||q^{(1)}||_{L^2(\rd{x}\rd{v})}$ is bounded by $C(||  q^{(0)}_i||_{L^2(\rd{x}\rd{v})} + ||\partial_t q^{(0)}_i||_{L^2(\rd{x}\rd{v})})t$, see the proof of Theorem 4.1. 
Assume now that the statement is true for a fixed $n \in \mathbb{N}$. We want to deduce that it is true for $n+1$. If $q^{(n+1)}$ satisfies 
\begin{equation}
\partial_t q^{(n+1)} + \sum_{k=0}^{n+1} \begin{pmatrix}
n+1 \\ k
\end{pmatrix} \partial_z^{(n+1-k)} \left(\sqrt{T_\ast}\right) v \partial_x q^{(k)} = \mathcal{L}_1 q^{(n+1)}
\end{equation}
we can replace $v \partial_x q^{(n)}$ in terms of $\partial_t q^{(n)}, \mathcal{L}_1 q^{(n)}, v \partial_x q^{(N)}, N <n$ from the equation for $q^{(n)}$ given by \eqref{eqn:q^n}. In the resulting equation we can replace $v \partial_x q^{(n-1)}$ in terms of $\partial_t q^{(n-1)}, \mathcal{L}_1 q^{(n-1)}, v \partial_x q^{(N)}, N<n-1$ from the equation for $q^{(n-1)}$. Next, we can replace $v \partial_x q^{(n-2)}$ from the equation for $q^{(n-2)}$ and so on until we do not have terms with $v \partial_x q^{(k)}$ for some $k<n+1$ any more. So all in all, we obtain an equation of the form
\begin{equation}
\partial_t q^{(n+1)} + A(\partial_t q^{(0)}, \mathcal{L}_1 q^{(0)}, \dots , \partial_t q^{(n)}, \mathcal{L}_1 q^{(n)}, T_\ast ) + v \partial_x q^{(n+1)} = \mathcal{L}_1 q^{(n+1)}
\label{eqn:q^n+1}
\end{equation}
where $A$ is a linear combination of $\partial_t q^{(0)}, \mathcal{L}_1 q^{(0)}, \dots , \partial_t q^{(n)}, \mathcal{L}_1 q^{(n)}$ with coefficients depending on $T_\ast$ of the form 
\begin{equation}
\frac{(\partial_z^a(\sqrt{T_\ast}))^b}{\sqrt{T_\ast}^c} \quad \text{for} ~ a,b,c \leq n+1
\end{equation}
We can show that $\partial_t q^{(N)}, N \leq n$ satisfy the same equation as $q^{(N)}$ similar as it is done in section 3 for $\partial_x g$ and $g$ and $\mathcal{L}_1 q^{(N)}, N \leq n$ is bounded in ${L^2(\rd{x}\rd{v})}$ by $|| q^{(N)}||_{L^2(\rd{x}\rd{v})}$, and that they are bounded in $L^2(\rd{x}\rd{v})$ by $C_N t^N$ where $C_N$ depends on $||\mathcal{L}_1^k \partial_t^l q^{(0)}||_{L^2(\rd{x}\rd{v})}, k,l \leq N$ due to the induction assumption. Finally, by the energy method we can deduce from \eqref{eqn:q^n+1} that $q^{(n+1)}$ is bounded in $L^2(\rd{x}\rd{v})$ by $C_{n+1} t^{n+1}$.
$\hfill\Box$
\end{proof}

\end{document}